\documentclass[11pt,a4paper,leqno]{amsart}

\usepackage[latin1]{inputenc}
\usepackage[T1]{fontenc}
\usepackage[english]{babel}
\usepackage{amsfonts}
\usepackage{amsmath}
\usepackage{amssymb}
\usepackage{eurosym}
\usepackage{mathrsfs}
\usepackage{palatino}
\usepackage{color}
\usepackage{esint}
\usepackage{url}
\usepackage{verbatim}

\usepackage{enumerate}
\allowdisplaybreaks

\newcommand{\R}{\mathbb{R}}
\newcommand{\C}{\mathbb{C}}

\newcommand{\N}{\mathbb{N}}

\newcommand{\Z}{\mathbb{Z}}

\newcommand{\calI}{\mathcal{I}}

\newcommand{\bla}{\big \langle}
\newcommand{\bra}{\big \rangle}

\numberwithin{equation}{section}

\newcommand{\ud}[0]{\,\mathrm{d}}

\newcommand{\esssup}[0]{\operatornamewithlimits{ess\,sup}}

\newcommand{\BMO}[0]{\operatorname{BMO}}
\newcommand{\bmo}[0]{\operatorname{bmo}}

\newcommand{\calD}[0]{\mathcal{D}}

\swapnumbers
\theoremstyle{plain}
\newtheorem{thm}[equation]{Theorem}
\newtheorem{lem}[equation]{Lemma}
\newtheorem{prop}[equation]{Proposition}
\newtheorem{cor}[equation]{Corollary}

\theoremstyle{definition}

\theoremstyle{remark}
\newtheorem{rem}[equation]{Remark}

\title[Some new weighted estimates on product spaces]{Some new weighted estimates on product spaces}

\author[E.\, Airta]{Emil Airta}
\address[E.A.]{Department of Mathematics and Statistics, University of Helsinki, P.O.B. 68, FI-00014 University of Helsinki, Finland}
\email{emil.airta@helsinki.fi}

\author[K.\ Li]{Kangwei Li}
\address[K.L.]{Center for Applied Mathematics, Tianjin University, Weijin Road 92, 300072 Tianjin, China}
\email{kli@tju.edu.cn}

\author[H.\ Martikainen]{Henri Martikainen}
\address[H.M.]{Department of Mathematics and Statistics, University of Helsinki, P.O.B. 68, FI-00014 University of Helsinki, Finland}
\email{henri.martikainen@helsinki.fi}

\author[E.\ Vuorinen]{Emil Vuorinen}
\address[E.V.]{Department of Mathematics and Statistics, University of Helsinki, P.O.B. 68, FI-00014 University of Helsinki, Finland}
\email{j.e.vuorin@gmail.com}

\makeatletter
\@namedef{subjclassname@2010}{%
  \textup{2010} Mathematics Subject Classification}
\makeatother


\thanks
{H.M. was supported by the Academy of Finland through the grants 294840 and 327271, and by the three-year research grant 75160010 of the University of Helsinki. E.A. was supported by the Academy of Finland through the grant 327271.
They are both members of the Finnish Centre of Excellence in Analysis and Dynamics Research supported by the Academy of Finland (project No. 307333).
}

\begin{document}

\begin{abstract}
We complete our theory of weighted $L^p(w_1) \times L^q(w_2) \to L^r(w_1^{r/p} w_2^{r/q})$ estimates
for bilinear bi-parameter Calder\'on--Zygmund operators under the assumption that $w_1 \in A_p$ and $w_2 \in A_q$
are bi-parameter weights. This is done by lifting a previous restriction on the class of singular integrals
by extending a classical result of Muckenhoupt and Wheeden regarding weighted $\BMO$ spaces to the product BMO setting. We use this extension of the Muckenhoupt-Wheeden result also to generalise some two-weight commutator estimates from bi-parameter
to multi-parameter. This gives a fully satisfactory Bloom type upper estimate for $[T_1, [T_2, \ldots [b, T_k]]]$, where
each $T_i$ can be a completely general multi-parameter Calder\'on--Zygmund operator.
\end{abstract}

\maketitle

\section{Introduction}
Singular integral operators (SIOs) are operators of the form
\begin{equation*}\label{eq:SIO}
  Tf(x)=\int_{\R^d}K(x,y)f(y)\ud y.
\end{equation*}
They include many important linear transformations that arise in the analysis connecting
geometric measure theory, partial differential equations, harmonic analysis and functional analysis. Classical one-parameter kernels are singular when $x=y$. Product
space theory (multi-parameter theory), on the other hand, is concerned with kernels whose singularities are more spread out. To get an idea, for $x,y \in \C = \R \times \R$, compare the one-parameter Beurling kernel $1/(x-y)^2$ with the bi-parameter kernel $1/[(x_1-y_1)(x_2-y_2)]$ -- the product of Hilbert kernels in both coordinate directions.

Multi-parameter SIOs arise naturally in applications involving a product type estimate. A simple example is given by the multiplier operators. A multiplier $m \colon \R \times \R \to \C$ satisfying $|\partial^{\alpha} m(\xi)| \lesssim |\xi|^{-|\alpha|}$ for all multi-indices $\alpha = (\alpha_1, \alpha_2)$ and $\xi = (\xi_1, \xi_2) \in \R^2 \setminus \{0\}$
gives rise to a convolution form one-parameter SIO
$T_m$ with $\widehat{T_m f}(\xi) = m(\xi) \widehat f(\xi)$. However, if $m$ only satisfies \emph{the less demanding} estimate $|\partial_{\xi_1}^{\alpha_1} \partial_{\xi_2}^{\alpha_2} m(\xi)| \lesssim |\xi_1|^{-|\alpha_1|} |\xi_2|^{-|\alpha|_2}$, we get a bi-parameter SIO $T_m$. For the classical linear multi-parameter theory and some of its original applications
see e.g. Fefferman--Stein \cite{FS} and Journ\'e \cite{Jo}.

On the other hand, a heuristic model of an $n$-linear SIO $T$ in $\R^d$ is obtained by setting
$$
T(f_1,\ldots, f_n)(x) = U(f_1 \otimes \cdots \otimes f_n)(x,\ldots,x), \qquad x \in \R^d,\, f_i \colon \R^d \to \C,
$$
where $U$ is a linear singular integral operator in $\R^{nd}$. See e.g. Grafakos--Torres \cite{GT} for the basic theory.
Multilinear SIOs arise naturally from applications to partial differential equations, complex function theory and ergodic theory, among others.
For instance, $L^p$ estimates for the homogeneous fractional derivative $D^{\alpha} f=\mathcal F^{-1}(|\xi|^{\alpha} \widehat f(\xi))$ of a product of two or more functions, often referred to as \emph{fractional Leibniz rules}, are widely employed in the study of dispersive equations. This started from the work of Kato and Ponce \cite{KP}, and such estimates descend from the multilinear H\"ormander-Mihlin multiplier theorem of Coifman-Meyer \cite{CM}. A variety of formulations may be found e.g. in Grafakos--Oh \cite{GO}.

Finally, multilinear multi-parameter estimates arise naturally in applications whenever
a multilinear phenomena, like the fractional Leibniz rules, are combined with product type estimates, such as those that arise when we want to take different \emph{partial} fractional derivatives. We refer to our recent work \cite{LMV} for a thorough background on the subject, its significance and for new developments.
In the product setting this multilinear theory of SIOs is historically significantly more limited than in the one-parameter setting. For example, in the one-parameter -- linear or multilinear -- setting, the following Calder\'on--Zygmund type principle is standard: if an SIO is bounded with some exponents, it is bounded with all eligible exponents. In the linear
bi-parameter setting such principles follow from \cite{Jo} or \cite{Ma1}, but they are already more involved.
In \cite{LMV} we finally were able to develop such general principles in the bilinear bi-parameter setting: 
simpler estimates in the Banach range ($r > 1$), imply boundedness in the full bilinear range $L^p \times L^q \to L^r$, $1/p + 1/q = 1/r$, $1 < p,q \le \infty$, $1/2 < r < \infty$, weighted estimates, mixed-norm estimates, and so on.

However, our weighted
estimates in \cite{LMV} (see Section \ref{sec:def}
for the definition of $A_p$ weights) still had the restriction that we
needed the cancellation $T(1,1) = 0$, and the same for the adjoints and partial adjoints.
It is easy to come up with singular integrals, where this cancellation does not hold: already a tensor product
of two bilinear one-parameter SIOs does not in general satisfy it. In this paper we remove this final restriction, which leads
to a complete and satisfactory theory for general bilinear bi-parameter SIOs.

The assumptions on objects like $T(1,1)$ have to do with $T1$ type arguments.
Already in the linear one-parameter setting the question of the boundedness of $T$ with a generic kernel is often best answered by so-called $T1$ theorems, where the action of the operator $T$ on the constant function $1$ plays a critical role, and it is only the
convolution kernels $K(x,y)=k(x-y)$, which are conveniently studied via formulae like
$
\widehat{Tf}(\xi) =\widehat{k}(\xi)\widehat{f}(\xi).
$
While the assumption $T1 = 0$ is more general than $T$ being of convolution form, it is morally related.

The reason why we care about weighted estimates is that, beyond their significant intrinsic interest,
they are of fundamental use in proving other (unweighted) estimates, like obtaining the full bilinear range of exponents
from a single tuple $(p_0, q_0, r_0)$
and proving vector-valued and mixed-norm estimates. This is due to the very powerful bilinear extrapolation results -- see e.g.
\cite{DU, GM, LMO, LMMOV, Nieraeth}. In the product setting this viewpoint is particularly useful as many of the classical one-parameter
tools are crudely missing. However, extrapolation is already very convenient in the one-parameter multilinear theory due to, for example,
the complicated nature of multilinear interpolation.

For example, we obtained a range of mixed-norm $L^{p_1}L^{p_2}$ estimates in \cite{LMV} via weighted estimates, and, thus, they required
the same restriction on the class of SIOs. See e.g. Di Plinio--Ou \cite{DO} for some previous mixed-norm estimates.
 With the understanding that a Calder\'on--Zygmund operator (CZO) is an SIO
satisfying natural $T1$ type assumptions, our improvement of \cite{LMV} now reads:
\begin{thm}\label{thm:main2}
Let $T$ be a bilinear bi-parameter CZO as defined in \cite{LMV}. Then we have the weighted estimate
\[
\|T(f,g)\|_{L^r(w)}\le C([w_1]_{A_p}, [w_2]_{A_q}) \|f\|_{L^p(w_1)}\|g\|_{L^q(w_2)}
\]
for all $1<p,q<\infty$ and $1/2<r<\infty$ with $1/r=1/p+1/q$, and for all bi-parameter weights $w_1\in A_p(\R^{n}\times \R^m)$, $w_2\in A_q(\R^{n}\times \R^m)$ with $w=w_1^{r/p}w_2^{r/q}$.
In the unweighted case we also have the mixed-norm estimates
$$
\|T(f, g)\|_{L^{r_1}(\R^n; L^{r_2}(\R^m))} \lesssim \|f\|_{L^{p_1}(\R^n; L^{p_2}(\R^m))}\|g\|_{L^{q_1}(\R^n; L^{q_2}(\R^m))}
$$
for all $1 < p_i, q_i \le \infty$ and $1/2 < r_i < \infty$ with $1/p_i + 1/q_i = 1/r_i$, except that if $r_2 < 1$ we have
to assume $\infty \not \in \{p_1, q_1\}$.
\end{thm}
Bilinear weights pose a problem with duality: notice e.g. that if $w_1, w_2 \in A_4$ then $w := w_1^{1/2} w_2^{1/2} \in A_4$, while
we need to work in $L^2(w)$.  This is a relevant problem and often makes bilinear bi-parameter weighted estimates different
and harder than in the linear case. For the linear estimates see Fefferman--Stein \cite{FS} and Fefferman \cite{RF1, RF2}, and the much more recent Holmes--Petermichl--Wick \cite{HPW} that is rooted on the modern dyadic-probabilistic methods \cite{Ma1}.
 In particular, already some linear paraproduct estimates depend on suitable $H^1$-$\BMO$ type duality arguments, and,
for this reason, we could not previously handle weighted estimates for certain model operators. This led to the restriction on the class
of CZOs. We now remove this restriction by developing some new theory for the product $\BMO$ space of Chang and Fefferman \cite{CF1, CF2}.
Theorem \ref{thm:main1} below is an extension of a classical result of Muckenhoupt and Wheeden \cite{MW} to the product BMO setting -- it says
that certain weighted product $\BMO$ spaces are actually the same as the unweighted product $\BMO$ space.
This gives a useful way to construct
objects in some genuinely weighted product $\BMO$ spaces by starting with an object in the unweighted product $\BMO$.
We prove this result of independent interest in its full generality, and apply its special case to prove our bilinear bi-parameter weighted estimates.

Finally, we present another application of Theorem \ref{thm:main1}. This concerns the recently hot topic of two-weight estimates
for commutators of SIOs. The classical result of Coifman--Rochberg--Weiss \cite{CRW}
showed that
$$
\|b\|_{\BMO} \lesssim \|[b,T]\|_{L^p \to L^p} \lesssim \|b\|_{\BMO}, \textup{ where } [b,T]f := bTf - T(bf),
$$
for a class of non-degenerate one-parameter SIOs $T$. Commutator estimates then e.g. yield
factorizations for Hardy functions, imply div-curl lemmas relevant in compensated compactness, and are connected to
the Jacobian problem $Ju = f$ in $L^p$ (see Hyt\"onen \cite{HyCom}). The two-weight problem concerns
estimates from $L^p(\mu)$ to $L^p(\lambda)$ for two different weights $\mu, \lambda$ and has attracted
significant interest after the recent work by Holmes--Lacey--Wick \cite{HPW}. For the product space
versions of these two-weight estimates see the recent works \cite{EA, HPW, LMV:Bloom, LMV:Bloom2}.
In this paper we remove the final restriction on the most general estimate thus far, \cite[Theorem 1.5]{EA}, and prove a two-weight upper bound 
on $[T_1, [T_2, \ldots [b, T_k]]]$, where each $T_i$ can be a completely general multi-parameter CZO as in \cite{Ou}.
Previously, in certain situations it was explicitly required in \cite{EA} that $T_i$ can be at most bi-parameter, or else
it needs to satisfy some extra cancellation.
\begin{thm}\label{thm:main3}
Let $\R^d = \prod_{i=1}^m \R^{d_i}$ be a product space of $m$-parameters. Let $k \le m$ be given and $\calI = \{\calI_1, \ldots, \calI_k\}$ be a partition of $\{1, \ldots, m\}$.
Let $k_i = \# \calI_i$, and for each each $i = 1, \ldots, k$ let $T_i$ be a multi-parameter CZO as in \cite{Ou}, which is defined in the product space $\prod_{j \in \calI_i} \R^{d_j}$
and has $k_i$ parameters. Let $b \colon \R^d \to \C$, $p \in (1, \infty)$, $\mu, \lambda \in A_p(\R^d)$ be $m$-parameter weights and $\nu = \mu^{1/p} \lambda^{-1/p}$ be the
associated Bloom weight. Then we have
$$
\| [T_1, [T_2, \ldots [b, T_k]]] \|_{L^p(\mu) \to L^p(\lambda)} \lesssim \|b\|_{\bmo^{\calI}(\nu)},
$$
where in this formula every $T_i$ is extended to act on the whole product space $\R^d$, and $\bmo^{\calI}(\nu)$ is a suitable weighted little product BMO space
as in \cite{EA}.
\end{thm}
 We manage this extension in Section \ref{sec:MulBloom} by using the full power of Theorem \ref{thm:main1}. For the various definitions 
see Section \ref{sec:def} and Section \ref{sec:MulBloom}.
 
\subsection{In-depth discussion}\label{sec:indepth}
In \cite[Theorem 1.5]{EA} two-weight commutator 
estimates were shown for $$[T_1, [T_2, \ldots [b, T_k]]],$$ where each $T_i$ is a multi-parameter
CZO as in Ou \cite{Ou} (a multi-parameter dyadic representation theorem generalising \cite{Ma1}).
However, it was required that either $T_i$ satisfies some cancellation (is paraproduct free)
or is of at most two parameters. In Theorem \ref{thm:main3} we remove the restrictions completing the upper bound theory of multi-parameter
two-weight commutator bounds. We stress that this is philosophically different to what we do in Theorem \ref{thm:main2}
-- there we aim for complete \emph{bi-parameter} theory in the strictly different bilinear context by removing some previous cancellation restrictions present in the bilinear bi-parameter theory \cite{LMV}. In \cite{EA} no restrictions appear on the bi-parameter case to begin with -- restrictions are needed only if some of the operators $T_i$ are $m$-parameter, $m \ge 3$. 

The passage from bi-parameter to multi-parameter is known to be very non-trivial in some occasions. In \cite{EA} difficulties arose 
as certain so-called partial paraproduct terms (model operators appearing in the dyadic representation given by \cite{Ou}) seemed to require a sparse domination type treatment already
used in \cite{LMV:Bloom}, and this could only be done if the paraproduct component of the partial paraproduct had only
one parameter. In tri-parameter theory (and general multi-parameter theory) it was then necessary to assume that such partial paraproducts vanish (that is, the SIOs have
some extra cancellation), as otherwise partial paraproducts with multi-parameter paraproduct components arise.

All of this is somehow connected to the very innocent result of \cite[Lemma 6.7]{LMV}.
This particular one-parameter lemma is proved using sparse domination
and can be extrapolated to get useful vector-valued estimates, which we need in various places in \cite{LMV, LMV:Bloom}.
We now know how to prove some strong enough analogs
of \cite[Lemma 6.7]{LMV} without sparse domination and using Theorem \ref{thm:main1} instead, and this is extendable
to multi-parameter. See Lemma \ref{lem:mainLem} below. This allows us to remove the restrictions present in \cite{EA}.
The reader might still reasonably wonder would it now be easy enough to state a multi-parameter version of Theorem \ref{thm:main2} as well.
The short answer is that this might still require more work. This is because the linear two-weight theory only
requires some special versions of \cite[Lemma 6.7]{LMV}, where averages are outside as in Lemma \ref{lem:mainLem}, but
the weighted estimates for bilinear bi-parameter partial paraproducts \cite[Proposition 6.11]{LMV} use all of the symmetries of \cite[Lemma 6.7]{LMV}.

\section{Notations and preliminaries}\label{sec:def}
\subsection*{Basic notation}
Throughout this paper $A\lesssim B$ means that $A\le CB$ with some constant $C$ that we deem unimportant to track at that point.
In particular, we often do not track the dependence on the weight constants.
We write $A\sim B$ if $A\lesssim B\lesssim A$.

Before Section \ref{sec:MulBloom} we work in the bi-parameter setting in the product space $\R^{n+m}=\R^n\times \R^m$. We write $x=(x_1, x_2)$ with $x_1\in \R^n$ and $x_2\in \R^m$ and for $f: \R^{n+m}\to \mathbb C$ and $h:\R^n\to\mathbb C$, we define
\[
\langle f, h\rangle_1(x_2):=\int_{\R^n} f(x_1, x_2)h(x_1) \ud x_1.
\]
In Section \ref{sec:MulBloom}, which is the only place in the paper where we do multi-parameter theory (as opposed to bi-parameter),
some additional notation will be introduced.

\subsection*{Haar functions}
We denote a dyadic grid in $\R^n$ by $\calD^n$ and a dyadic grid in $\R^m$ by $\calD^m$. We often write $\calD = \calD^n \times \calD^m$
for the related dyadic rectangles.

For an interval $I \subset \R$ we denote by $I_{l}$ and $I_{r}$ the left and right
halves of $I$, respectively. We define $h_{I}^0 = |I|^{-1/2}1_{I}$ and $h_{I}^1 = |I|^{-1/2}(1_{I_{l}} - 1_{I_{r}})$.
Let now $Q = I_1 \times \cdots \times I_n \in \calD^n$, and define the Haar function $h_Q^{\eta}$, $\eta = (\eta_1, \ldots, \eta_n) \in \{0,1\}^n$, by setting
\begin{displaymath}
h_Q^{\eta} = h_{I_1}^{\eta_1} \otimes \cdots \otimes h_{I_n}^{\eta_n}.
\end{displaymath}
If $\eta \ne 0$ the Haar function is cancellative: $\int h_Q^{\eta} = 0$. We exploit notation by suppressing the presence of $\eta$, and write $h_Q$ for some $h_Q^{\eta}$, $\eta \ne 0$. If $R = I \times J \in \calD = \calD^n \times \calD^m$, we define
$h_R = h_I \otimes h_J$.

\subsection*{Weights}
A weight $w(x_1, x_2)$ (i.e. a locally integrable a.e. positive function) belongs to the bi-parameter $A_p(\R^n \times \R^m)$, $1 < p < \infty$, if
$$
[w]_{A_p(\R^n \times \R^m)} := \sup_{R} \frac 1{|R|}\int_R w   \Bigg( \frac 1{|R|}\int_R w^{1-p'}\Bigg)^{p-1} < \infty,
$$
where the supremum is taken over $R = I \times J$, where $I \subset \R^n$ and $J \subset \R^m$ are cubes
with sides parallel to the axes (we simply call such $R$ rectangles).
We have
$$
[w]_{A_p(\R^n\times \R^m)} < \infty \textup { iff } \max\big( \esssup_{x_1 \in \R^n} \,[w(x_1, \cdot)]_{A_p(\R^m)}, \esssup_{x_2 \in \R^m}\, [w(\cdot, x_2)]_{A_p(\R^n)} \big) < \infty,
$$
and that $\max\big( \esssup_{x_1 \in \R^n} \,[w(x_1, \cdot)]_{A_p(\R^m)}, \esssup_{x_2 \in \R^m}\, [w(\cdot, x_2)]_{A_p(\R^n)} \big) \le [w]_{A_p(\R^n\times \R^m)}$, while
the constant $[w]_{A_p}$ is dominated by the maximum to some power.
We say $w\in A_\infty(\R^n\times \R^m)$ if
\[
[w]_{A_\infty(\R^n\times \R^m)}:=\sup_R \frac 1{|R|}\int_R w  \exp\Bigg( \frac1{|R|}\int_R \log (w^{-1})  \Bigg)<\infty.
\]
It is well-known (see e.g. \cite[Section 7]{DMO}) that
$$A_\infty(\R^n\times \R^m)=\bigcup_{1<p<\infty}A_p(\R^n\times \R^m).$$
We do not have any explicit use for the $A_{\infty}$ constant. The $w \in A_{\infty}$ assumption
can always be replaced with the explicit assumption $w \in A_s$ for some $s \in (1,\infty)$, and
then estimating everything with a dependence on $[w]_{A_s}$.

Of course, $A_p(\R^n)$ is defined similarly as $A_p(\R^n \times \R^m)$ -- just take the supremum over cubes $Q$.
A modern reference for the basic theory of bi-parameter weights is e.g. \cite{HPW}.

\subsection*{Square functions and maximal functions}
Given $f \colon \R^{n+m} \to \C$ and $g \colon \R^n \to \C$ we denote the dyadic maximal functions
by
$$
M_{\calD^n}g := \sup_{I \in \calD^n}  1_I \bla |g|\bra_I \, \textup{ and } \,
M_{\calD} f:= \sup_{R \in \calD}    1_R \bla |f|\bra_R,
$$
where $\bla f \bra_A = |A|^{-1} \int_A f$. We can e.g. write $\bla f \bra_I^1$ if $I \subset \R^n$ and we average only over the first parameter.
We also set $M^1_{\calD^n} f(x_1, x_2) =  M_{\calD^n}(f(\cdot, x_2))(x_1)$. The operator $M^2_{\calD^m}$ is defined similarly. The weighted maximal function is defined by
\[
M^w_{\calD} f:= \sup_{  R\in \calD}1_R \bla |f|\bra_R^w,
\]
where $\bla |f|\bra_R^w = w(R)^{-1} \int_R |f|w$. We require the following very nice result of Fefferman \cite{RF}.
For a modern reference see \cite[Proposition B.1]{LMV:Bloom}. Notice that in the bi-parameter setting this result is very non-trivial
as $w$ is not of tensor form.
\begin{lem}\label{lem:wm}
Let $w\in A_\infty(\R^n\times \R^m)$. Then for all $1<p\le \infty$ we have
\[
\|M^w_{\calD} f\|_{L^p(w)}\lesssim  \|f\|_{L^p(w)}.
\]
\end{lem}
Now define the square functions
$$
S_{\calD} f = \Bigg( \sum_{R \in \calD}  \big|\bla f, h_R\bra\big|^2 \frac{1_R}{|R|} \Bigg)^{1/2}, \,\, S_{\calD^n}^1 f =  \Bigg( \sum_{I \in \calD^n}  \frac{1_I}{|I|} \otimes \big|\bla f, h_I\bra_1\big|^2 \Bigg)^{1/2}
$$
and define $S_{\calD^m}^2 f$ analogously.
Define also
$$
S_{\calD, M}^1 f = \Bigg( \sum_{I \in \calD^n} \frac{1_I}{|I|} \otimes \big[M_{\calD^m} \bla f, h_I \bra_1\big]^2 \Bigg)^{1/2}, \,\, S_{\calD, M}^2 f = \Bigg( \sum_{J \in \calD^m} \big[M_{\calD^n} \bla f, h_J \bra_2\big]^2 \otimes \frac{1_J}{|J|}\Bigg)^{1/2}.
$$
We record the following basic weighted estimates, which are used repeatedly below.
\begin{lem}\label{lem:standardEst}
For $p \in (1,\infty)$ and $w \in A_p = A_p(\R^n \times \R^m)$ we have the weighted square function estimates
$$
\| f \|_{L^p(w)}
 \sim  \| S_{\calD} f\|_{L^p(w)}
\sim   \| S_{\calD^n}^1 f  \|_{L^p(w)}
\sim  \| S_{\calD^m}^2 f  \|_{L^p(w)}.
$$
Moreover, for $p, s \in (1,\infty)$ we have the Fefferman--Stein inequality
$$
\Bigg\| \Bigg( \sum_j |M f_j |^s \Bigg)^{1/s} \Bigg\|_{L^p(w)} \lesssim  \Bigg\| \Bigg( \sum_{j} | f_j |^s \Bigg)^{1/s} \Bigg\|_{L^p(w)}.
$$
Here $M$ can e.g. be $M_{\calD^n}^1$ or $M_{\calD}$. Finally, we have
$$
\| S_{\calD, M}^1 f\|_{L^p(w)} + \| S_{\calD, M}^2 f\|_{L^p(w)} \lesssim  \|f\|_{L^p(w)}.
$$
\end{lem}
See \cite[Lemma 2.1]{LMV:Bloom} for an indication on how to prove this standard result.
Of key importance to us is the following lower square function estimate valid for $A_{\infty}$ weights.
\begin{lem}\label{lem:LowerFS} There holds
$$
\|f\|_{L^p(w)} \lesssim \|S_{\calD^n}^1 f\|_{L^p(w)}
$$
and
$$
\|f\|_{L^p(w)} \lesssim \|S_{\calD} f\|_{L^p(w)}
$$
for all $p \in (0, \infty)$ and bi-parameter weights $w \in A_{\infty}$.
\end{lem}
See \cite[Section 6]{ LMV} for an explanation of this well-known inequality. This is important to us
as the weight $w=w_1^{r/p}w_2^{r/q}$ in Theorem \ref{thm:main2} is at least $A_{\infty}$ -- it is in fact $A_{2r}$.

\subsection*{Singular integrals}
A kernel $K \colon \R^d \times \R^d \setminus \{(x,y) \in \R^d \times \R^d \colon x =y \} \to \C$ is a standard Calder\'on-Zygmund kernel on $\R^d$
if we have
$$
|K(x,y)| \le \frac{C}{|x-y|^d}
$$
and, for some $\alpha \in (0,1]$,
\begin{equation*}
|K(x, y) - K(x', y)| + |K(y, x) - K(y, x')| \le C\frac{|x-x'|^{\alpha}}{|x-y|^{d+\alpha}}
\end{equation*}
whenever $|x-x'| \le |x-y|/2$. 

A singular integral operator (SIO) is a linear operator $T$ (initially defined, for example, on bounded and compactly supported functions) so that there is a standard kernel
$K$ for which
$$
\langle Tf ,g \rangle = \iint_{\R^d \times \R^d} K(x,y) f(y) g(x) \ud y \ud x
$$
whenever the functions $f$ and $g$ are nice and have disjoint supports. A Calder\'on--Zygmund operator (CZO) is an SIO $T$, 
which satisfies the $T1$ condition
$$
\int_I |T1_I| + \int_I |T^*1_I| \lesssim |I|
$$
for all cubes $I \subset \R^d$. Here $T^*$ is the linear adjoint of $T$. A $T1$ theorem says that an SIO is a CZO if and only if
it is bounded from $L^p \to L^p$ for all (equivalently for some) $p \in (1,\infty)$. 
We know a lot about the structure of a CZO $T$.
Indeed, we can represent $T$ with certain dyadic model operators (DMOs) -- see \cite{Hy1, Hy2}. The DMOs
take the concrete form of so-called dyadic shifts and paraproducts. Moreover, the $T1$ theorem is a consequence of the representation theorem.
We will indicate how a representation theorem looks soon.

A model of a bi-parameter singular integral operator in $\R^n \times \R^m$ is $T_1 \otimes T_2$, where
$T_1$ and $T_2$ are usual singular integrals in $\R^n$ and $\R^m$, respectively.
The general definition of a bi-parameter singular integral $T$ requires that 
$\langle Tf_1, f_2\rangle$, $f_i = f_i^1 \otimes f_i^2$, can be written using different kernel representations depending on whether
\begin{enumerate}
\item $\operatorname{spt} f_1^1 \cap \operatorname{spt}f_2^1 = \emptyset$ and $\operatorname{spt} f_{1}^2 \cap \operatorname{spt} f_{2}^2 = \emptyset$,
\item $\operatorname{spt} f_1^1 \cap \operatorname{spt}f_2^1 = \emptyset$ or
\item $\operatorname{spt} f_{1}^2 \cap \operatorname{spt}f_{2}^2 = \emptyset$.
\end{enumerate}
In the first case we have a so-called full kernel representation, while in cases $2$ and $3$
a partial kernel representations holds in $\R^n$ or $\R^m$, respectively. The $T1$ conditions are more complicated than in one-parameter -- for the complete definitions see \cite[Section 2]{Ma1}. As in the one-parameter case, we can represent a bi-parameter CZO -- a bi-parameter SIO satisfying the bi-parameter $T1$ conditions -- using bi-parameter DMOs.
We will only need these model operators in this paper, and they are defined when they are needed.
Everything in Theorem \ref{thm:main2} and Theorem \ref{thm:main3} is reduced, via representation theorems, to estimates of model operators.

The analogous multi-parameter theory is presented in \cite{Ou}. An inherent complication of these bi-parameter and multi-parameter representations is the presence not only of ``pure'' paraproducts and cancellative shifts, but also of their hybrid combinations -- partial paraproducts -- that are completely new compared to the one-parameter case. 

For Theorem \ref{thm:main2} we still need to discuss the bilinear theory.
 A heuristic model of a bilinear one-parameter SIO $T$ in $\R^d$ is
$T(f_1, f_2)(x) :=  U(f_1 \otimes f_2)(x,x)$, where $x \in \R^d$, $f_i \colon \R^d \to \C$, $(f_1 \otimes f_2)(x_1,x_2) = f_1(x_1) f_2(x_2)$
 and $U$ is a linear SIO in $\R^{2d}$.
In detail, a bilinear SIO $T$ has a kernel $K$ satisfying estimates that are obtained from the above heuristic via the linear estimates, and if $\operatorname{spt} f_i \cap \operatorname{spt} f_j = \emptyset$ for some $i,j$ then
$$
\langle T(f_1, f_2), f_3 \rangle =  \iiint_{\R^{3d}} K(x,y,z) f_1(y) f_2(z) f_3(x) \ud y \ud z \ud x.
$$

Finally, a model of a bilinear bi-parameter singular integral in $\R^n \times \R^m$ is
$$
(T_n \otimes T_m)(f_1 \otimes f_2, g_1 \otimes g_2)(x) := T_n(f_1, g_1)(x_1)T_m(f_2, g_2)(x_2),
$$
where $f_1, g_1 \colon \R^n \to \C$, $f_2, g_2 \colon \R^m \to \C$,
$x = (x_1, x_2) \in \R^{n+m}$ and $T_n$, $T_m$ are bilinear SIOs defined in $\R^n$ and $\R^m$, respectively.
See \cite[Section 3]{LMV} for the rather long definition in the full generality. We reiterate that we will only meet the DMOs in this paper.
This is because by \cite[Theorem 1.3]{LMV} we have that if
$T$ is a bilinear bi-parameter CZO, then
$$
\langle T(f_1,f_2), f_3\rangle = C_T \mathbb{E}_{\omega}\mathop{\sum_{k = (k_1, k_2, k_3) \in \Z_+^3}}_{v = (v_1, v_2, v_3) \in \Z_+^3} \alpha_{k, v} 
\bla U^{k,v}_{\omega}(f_1, f_2), f_3 \bra,
$$
where $\omega = (\omega_1, \omega_2)$ is associated to random dyadic grids $\calD_{\omega} = \mathcal{D}^n_{\omega_1} \times \mathcal{D}^m_{\omega_2}$,
 $C_T \lesssim 1$, the numbers $\alpha_{k, v} > 0$ decay exponentially in complexity $(k,v)$, and
$U^{k,v}_{\omega}$ denotes some bilinear bi-parameter dyadic model operator of complexity $(k,v)$
defined in the lattice $\mathcal{D}_{\omega}$. Crucially, we have shown the desired weighted estimates in \cite{LMV} for all but one type of 
model operators -- the remaining model operators are defined carefully in Section \ref{sec:bilinbiparproof} and the weighted estimate is proved for them.
This shows Theorem \ref{thm:main2}. 

\section{Weighted BMO Spaces}\label{sec:BMO}
Let $\calD = \calD^n \times \calD^m$ be a lattice of dyadic rectangles, $A = (a_R)_{R \in \calD}$ be a sequence of scalars and $\Omega \subset \R^{n+m}$. We define
$$
S_A(x) = \Bigg(\sum_{R \in \calD} |a_R|^2 \frac{1_R(x)}{|R|} \Bigg)^{1/2} \qquad \textup{and} \qquad
S_{A, \Omega}(x) = \Bigg(\sum_{ \substack{ R \in \calD \\ R \subset \Omega}} |a_R|^2 \frac{1_R(x)}{|R|} \Bigg)^{1/2}.
$$
Let $p \in (0,\infty)$. Define
$$
\|A\|_{\BMO_{\textup{prod}}(p)} = \sup_{\Omega} \frac{1}{|\Omega|^{1/p}} \|S_{A, \Omega}\|_{L^p},
$$
where $\Omega$ is open and $0 < |\Omega| < \infty$.
There are many possibilities how to define a weighted version. The following is not the `correct' definition for many things. Nonetheless,
it will be of key use to us.
Thus, let $w \in A_{\infty}$ and set
$$
\|A\|_{\BMO_{\textup{prod},w}(p)} = \sup_{\Omega} \frac{1}{w(\Omega)^{1/p}} \|S_{A, \Omega}\|_{L^p(w)}.
$$
We set $\|A\|_{\BMO_{\textup{prod}}} = \|A\|_{\BMO_{\textup{prod}}(2)}$ and $\|A\|_{\BMO_{\textup{prod}, w}}
= \|A\|_{\BMO_{\textup{prod}, w}(2)}$.
The weight turns out not to play a role here -- that is, we have
$\|A\|_{\BMO_{\textup{prod}}} = \|A\|_{\BMO_{\textup{prod}, w}}$ for all bi-parameter weights $w \in A_{\infty}$. To prove this
we need the bi-parameter John-Nirenberg.
The unweighted version $\|A\|_{\BMO_{\textup{prod}}} \sim \|A\|_{\BMO_{\textup{prod}}(p)}$ is well-known.
However, we need to know it in the following form, which is a priori stronger. The proof is similar, though, but
requires the very non-trivial Lemma \ref{lem:wm}.
\begin{prop}\label{prop:JN}
For all bi-parameter weights $w \in A_{\infty}$ we have
$$
\|A\|_{\BMO_{\textup{prod}, w}} \sim \|A\|_{\BMO_{\textup{prod}, w}(p)}, \qquad 0 < p < \infty.
$$
\end{prop}
\begin{proof}
By H\"older's inequality we only need to prove $\|A\|_{\BMO_{\textup{prod}, w}(q)} \lesssim \|A\|_{\BMO_{\textup{prod}, w}(p)}$ for
$p < q$ and $q > 2$. We may assume $a_R \ne 0$ for only finitely many $R \in \calD$.
We fix $\Omega$ and, for a large enough $N > 0$, denote
$
E:= \{ S_{A, \Omega} > N\|A\|_{\BMO_{\textup{prod}, w}(p)} \}.
$
We now have
\[
w(E)\le (N\|A\|_{\BMO_{\textup{prod}, w}(p)})^{-p}  \| S_{A, \Omega} \|_{L^p(w)}^p \le N^{-p} w(\Omega).
\]
Split $\calD = \mathcal R_1\cup \mathcal R_2$, where
\[
\mathcal R_1:= \{R\colon w(E\cap R)> w(R)/2\},\, \mathcal R_2:= \{R\colon w(E\cap R)\le w(R)/2\}.
\]
Notice that clearly for all $R\in \mathcal R_1$ we have
\[
R \subset \{M_{\calD}^w(1_E)> 1/2\}=:\tilde E.
\]
Since $w\in A_\infty$, by Lemma \ref{lem:wm} we have
\[
w(\tilde E)\lesssim \|M^w_{\calD}(1_E)\|_{L^2(w)}^2\lesssim w(E).
\]
We now fix $N$ so that we always have $w(\tilde E) \le w(\Omega)/2^q$, and then notice that
\begin{align*}
\Bigg\|\Bigg(\sum_{\substack{R\in \mathcal R_1 \\ R\subset \Omega}} |a_R|^2 \frac {1_R}{|R|}\Bigg)^{\frac 12} \Bigg\|_{L^q(w)}&\le
\|S_{A, \tilde E} \|_{L^q(w)} \\
& \le \|A\|_{\BMO_{\textup{prod}, w}(q)}w(\tilde E)^{\frac 1q}\le \frac 12 \|A\|_{\BMO_{\textup{prod}, w}(q)}w(\Omega)^{\frac 1q}.
\end{align*}
This is absorbable, so we now move on to consider the sum, where $R\in \mathcal R_2$. As $q > 2$ we may calculate
\begin{align*}
\Bigg\|\Bigg(\sum_{\substack{R\in \mathcal R_2 \\ R\subset \Omega}} |a_R|^2 \frac {1_R}{|R|}\Bigg)^{\frac 12} \Bigg\|_{L^q(w)}^2 &=\sup_{\|g\|_{L^{(q/2)'}(w)}=1} \Bigg| \sum_{\substack{R\in \mathcal R_2 \\ R\subset \Omega}} |a_R|^2 \bla gw\bra_R \Bigg| \\
&\le 2\sup_{\|g\|_{L^{(q/2)'}(w)}=1} \| 1_{E^c} S_{A, \Omega}^2 M_{\calD}^wg\|_{L^1(w)}\\
&\le 2 \sup_{\|g\|_{L^{(q/2)'}(w)}=1}\|1_{E^c} S_{A,\Omega} \|_{L^q(w)}^2\|M_{\calD}^wg\|_{L^{(q/2)'}(w)} \\
&\lesssim \|A\|_{\BMO_{\textup{prod}, w}(p)}^2 w(\Omega)^{2/q},
\end{align*}
where we used Lemma \ref{lem:wm} in the last step. The proof is done as we have shown that
$$
\|A\|_{\BMO_{\textup{prod}, w}(q)} \le  \frac 12 \|A\|_{\BMO_{\textup{prod}, w}(q)} +  C \|A\|_{\BMO_{\textup{prod}, w}(p)}.
$$
\end{proof}
\begin{thm}\label{thm:main1}
For all bi-parameter weights $w \in A_{\infty}$ we have
$$\|A\|_{\BMO_{\textup{prod}}} \sim \|A\|_{\BMO_{\textup{prod}, w}}.
$$
\end{thm}
\begin{proof}
Fix $w\in A_\infty$. Then there exists $s > 2$ so that $w\in A_s$.
We first prove $\|A\|_{\BMO_{\textup{prod}, w}}\lesssim \|A\|_{\BMO_{\textup{prod}}}$. Define the
linear bi-parameter paraproduct $$\Pi f = \Pi_A f =\sum_{R\in\calD} a_R \bla f\bra_R h_R.$$ It is well-known (see e.g.
\cite{HPW}) that
\[
\|\Pi f \|_{L^s(w)}\lesssim \|A\|_{\BMO_{\textup{prod}}}\|f\|_{L^s(w)}.
\]Then by Lemma \ref{lem:standardEst} we have
\begin{align*}
\Bigg\|\Bigg(\sum_{R\in\calD}|a_R|^2 \big|\bla f\bra_R\big|^2 \frac{1_R}{|R|} \Bigg)^{\frac 12}\Bigg\|_{L^s(w)}=\|S_\calD (\Pi f)\|_{L^s(w)}\lesssim \|A\|_{\BMO_{\textup{prod}}} \|f\|_{L^s(w)}.
\end{align*}Testing with $f=1_\Omega$ we get
\[
\|S_{A, \Omega} \|_{L^s(w)}\le \Bigg\|\Bigg(\sum_{R\in\calD}|a_R|^2 \big|\bla 1_\Omega\bra_R\big|^2 \frac{1_R}{|R|} \Bigg)^{\frac 12}\Bigg\|_{L^s(w)}\lesssim \|A\|_{\BMO_{\textup{prod}}}w(\Omega)^{\frac 1s}.
\]
This means that $\|A\|_{\BMO_{{\rm{prod}},w}(s)}\lesssim \|A\|_{\BMO_{\textup{prod}}}$. By Proposition \ref{prop:JN} we conclude that $\|A\|_{\BMO_{\textup{prod}, w}}\lesssim \|A\|_{\BMO_{\textup{prod}}}$.

It remains to prove $\|A\|_{\BMO_{\textup{prod}}}\lesssim \|A\|_{\BMO_{\textup{prod}, w}}$. For $0\le f\in L^s(w)$ and $0\le g\in L^{(s/2)'}(w)$ , we have
\begin{align*}
\sum_{R\in\calD}|a_R|^2\bla w\bra_R \bla f\bra_R^2\bla g\bra_R^w&= \int_0^\infty \sum_{\substack{R\in\calD\\ \langle f\rangle_R^2\langle g\rangle_R^w>t}}|a_R|^2\bla w\bra_R \ud t\\
&\le \int_0^\infty \sum_{\substack{R\in\calD\\ R\subset \{(M_{\calD}f)^2 M_{\calD}^w g>t\}}}|a_R|^2\bla w\bra_R \ud t\\
&\le  \|A\|_{\BMO_{\textup{prod}, w}}^2\int_0^\infty w(\{(M_{\calD}f)^2 M_{\calD}^w g>t\})  \ud t\\
&=\|A\|_{\BMO_{\textup{prod}, w}}^2 \|(M_{\calD}f)^2 M_{\calD}^w g\|_{L^1(w)}\\
&\le \|A\|_{\BMO_{\textup{prod}, w}}^2 \|M_{\calD}f\|_{L^s(w)}^2\|M_{\calD}^w g\|_{L^{(s/2)'}(w)}\\
&\lesssim \|A\|_{\BMO_{\textup{prod}, w}}^2 \|f\|_{L^s(w)}^2\|  g\|_{L^{(s/2)'}(w)},
\end{align*}where we have used Lemma \ref{lem:wm} in the last step. Testing the above inequality with $f=w^{-\frac 1s}1_\Omega$, $g=w^{-\frac 1{(s/2)'}}1_\Omega$ we get
\[
\sum_{\substack{R\in\calD\\R\subset \Omega}}|a_R|^2  \bla w^{-\frac 1s}\bra_R^2\bla w^{\frac 2s}\bra_R\lesssim \|A\|_{\BMO_{\textup{prod}, w}}^2 |\Omega|.
\]We conclude the proof by noticing that
$
1\le \bla w^{-\frac 1s}\bra_R^2\bla w^{\frac 2s}\bra_R.
$
\end{proof}
\begin{rem}
In the one-parameter case, the equivalence between $\BMO$ and $\BMO_w$, where $w\in A_\infty$, is due to Muckenhoupt and Wheeden \cite{MW}.
\end{rem}

Finally, we define the actual weighted product BMO by setting
$$
\|A\|_{\BMO_{\textup{prod}}(w)} = \sup_{\Omega} \frac{1}{w(\Omega)^{1/2}} \Bigg\| \Bigg(\sum_{ \substack{ R \in \calD \\ R \subset \Omega}} |a_R|^2 \frac{1_R(x)}{w(R)} \Bigg)^{1/2} \Bigg\|_{L^2(\ud x)}
= \sup_{\Omega}\Bigg(\frac{1}{w(\Omega)}\sum_{\substack{R\in \calD\\ R\subset \Omega}}  \frac {|a_R|^2}{\bla w \bra_R}\Bigg)^{\frac 12}.
$$
The previous theorem is of independent interest, but also yields the following key lemma.
\begin{lem}\label{lem:main}
If $A \in \BMO_{\textup{prod}}$ define $A_w = ( a_R \langle w \rangle_R )_{R \in \calD}$ for $w \in A_{\infty}$. Then we have
$$
\|A_w\|_{\BMO_{\textup{prod}}(w)} \sim \|A\|_{\BMO_{\textup{prod}}}.
$$
\end{lem}
\begin{proof}
Notice that
$$
\|A_w\|_{\BMO_{\textup{prod}}(w)} = \|A\|_{\BMO_{\textup{prod}, w}} \sim \|A\|_{\BMO_{\textup{prod}}}.
$$
Here the first equality is obvious and the second estimate is Theorem \ref{thm:main1}.
\end{proof}
\begin{cor}\label{cor:main}
For sequences of scalars $A = (a_R)$ and $B = (b_R)$ we have
$$
\sum_{R \in \calD} |a_R| \langle w \rangle_R |b_R| \lesssim  \|A\|_{\BMO_{\textup{prod}}} \|S_B\|_{L^1(w)}
$$
whenever $w \in A_{\infty}$.
\end{cor}
\begin{proof}
Follows from the known, see e.g.  \cite[Proposition 4.1]{HPW}, weighted $H^1$-$\BMO$ duality
\begin{equation}\label{eq:standardH1BMO}
\sum_{R \in \calD} |a_R|  |b_R| \lesssim  \|A\|_{\BMO_{\textup{prod}}(w)} \|S_B\|_{L^1(w)}
\end{equation}
and Lemma \ref{lem:main}.
\end{proof}

\section{Proof of Theorem \ref{thm:main2}}\label{sec:bilinbiparproof}
A bilinear bi-parameter full paraproduct on a grid $\calD = \calD^n \times \calD^m$ has the form
\begin{equation}\label{eq:OneFullPara}
\Pi_A(f_1, f_2) = \Pi(f_1, f_2) = \sum_{R = I \times J \in \calD}
a_R
\Big \langle f_1, h_I \otimes \frac{1_J}{|J|} \Big \rangle
\bla f_2 \bra_{R}
\frac{1_I}{|I|} \otimes h_J,
\end{equation}
where $\|A\|_{\BMO_{\textup{prod}}} \le 1.$ What is important is that there are actually nine different types of full paraproducts -- the full paraproduct above corresponds to the tuples $\Big( h_I, \frac{1_I}{|I|}, \frac{1_I}{|I|} \Big)$ and
$\Big( \frac{1_J}{|J|}, \frac{1_J}{|J|}, h_J\Big)$, but the $h_I$ can be in any of the three slots and so can the $h_J$.

It follows from \cite{LMV} that to prove Theorem \ref{thm:main2} it suffices to prove the following weighted estimate
for the full paraproducts.
\begin{prop}
Let $1 < p, q < \infty$ and $1/2 < r < \infty$ satisfy $1/p+1/q = 1/r$, $w_1 \in A_p$ and $w_2 \in A_q$
be bi-parameter weights, and set $w := w_1^{r/p} w_2^{r/q}$.
Then we have
$$
\|\Pi(f_1, f_2)\|_{L^r(w)} \lesssim \|f_1\|_{L^p(w_1)} \|f_2\|_{L^q(w_2)}.
$$
\end{prop}
\begin{proof}
\textbf{Case 1.} Suppose that there is a full average over $R = I \times J$ at least in $f_1$ or $f_2$.
In such cases the bilinear paraproduct estimate decouples reducing to linear estimates.
For example, suppose $\Pi$
has the form \eqref{eq:OneFullPara}. Then using the weighted lower square function estimate, Lemma \ref{lem:LowerFS}, and
the basic Lemma \ref{lem:standardEst} we have
\begin{align*}
\big\|\Pi(f_1,f_2)\big\|_{L^r(w)}&\lesssim \Bigg\| \Bigg(\sum_{J} \Bigg|\sum_{I}|a_R| \Big\langle f_1, h_I \otimes \frac{1_J}{|J|}\Big\rangle \bla f_2\bra_{I\times J} \frac{1_I}{|I|} \Bigg|^2 \frac{1_J}{|J|} \Bigg)^{\frac 12}\Bigg\|_{L^r(w)}\\
&\lesssim \Bigg\|M_{\calD}f_2 \Bigg(\sum_{J} \Bigg(\sum_{I}|a_R| \Big|\Big\langle f_1, h_I \otimes \frac{1_J}{|J|}\Big\rangle \Big|  \frac{1_I}{|I|} \Bigg)^2 \frac{1_J}{|J|} \Bigg)^{\frac 12}\Bigg\|_{L^r(w)}\\
&\le \|M_{\calD}f_2\|_{L^q(w_2)}\Bigg\| \Bigg(\sum_{J} \Bigg(\sum_{I}|a_R|  \Big|\Big\langle f_1, h_I \otimes \frac{1_J}{|J|}\Big\rangle \Big|  \frac{1_I}{|I|} \Bigg)^2 \frac{1_J}{|J|} \Bigg)^{\frac 12}\Bigg\|_{L^p(w_1)}\\
&\lesssim \|f_2\|_{L^q(w_2)} \| S_{\calD^m}^2h \|_{L^p(w_1)},
\end{align*}
where
\[
h=\sum_{R = I \times J} |a_R|  \Big|\Big\langle f_1, h_I \otimes \frac{1_J}{|J|}\Big\rangle \Big|\frac{1_I}{|I|}\otimes h_J.
\]This is just a standard linear bi-parameter paraproduct, and thus satisfies the weighted estimate $\|h\|_{L^p(w_1)}\lesssim \|f_1\|_{L^p(w_1)}$ (see e.g.
\cite{HPW}).
Thus, we are done by Lemma \ref{lem:standardEst}.

\textbf{Case 2.}
Out of the remaining cases we choose the symmetry
\[
\Pi(f_1,f_2)= \sum_{R = I \times J} a_R \Big \langle f_1, h_I \otimes \frac{1_J}{|J|}\Big \rangle \Big \langle f_2,  \frac{1_I}{|I|}\otimes h_J\Big\rangle \frac{1_{R}}{|R| }.
\]
Equipped with our current tools we can prove the desired estimate directly
for any $p_0, q_0 \in (1,\infty)$ and $r_0 \in [1,\infty)$ satisfying $1/r_0 = 1/p_0 + 1/q_0$. By
bilinear extrapolation \cite{DU, GM} it is enough to prove the estimate with only one fixed tuple, so this is certainly enough to get the claimed full range.
For example, in the case $r_0 = 1$ we get
\begin{align*}
\|\Pi(f_1,f_2)\|_{L^1(w)} &\le  \sum_{R = I \times J} |a_R|\bla w \bra_R \Big| \Big \langle f_1, h_I \otimes \frac{1_J}{|J|}\Big \rangle \Big| \Big| \Big \langle f_2,  \frac{1_I}{|I|}\otimes h_J\Big\rangle \Big| \\
&\lesssim \Bigg\| \Bigg( \sum_{R = I \times J}\Big| \Big \langle f_1, h_I \otimes \frac{1_J}{|J|}\Big \rangle \Big|^2 \Big| \Big \langle f_2,  \frac{1_I}{|I|}\otimes h_J\Big\rangle \Big|^2 \frac{1_R}{|R|} \Bigg)^{1/2} \Bigg\|_{L^1(w)} \\
&\le  \| S_{\calD,M}^1 f_1\|_{L^{p_0}(w_1)} \|  S_{\calD,M}^2 f_2\|_{L^{q_0}(w_2)} \lesssim  \|  f_1\|_{L^{p_0}(w_1)} \|  f_2\|_{L^{q_0}(w_2)},
\end{align*}
where we have used Corollary \ref{cor:main} in the second estimate and Lemma \ref{lem:standardEst} in the last step.
\end{proof}
\begin{rem}
The advantage of the case $r_0 = 1$ in Case 2 above is that then $w \in A_2$, so that proving
the required estimate $\|A\|_{\BMO_{\textup{prod}, w}} \lesssim \|A\|_{\BMO_{\textup{prod}}}$ does not require
the John-Nirenberg inequality and thus not even Lemma \ref{lem:wm}.  However,
we still note that the case $r_0 > 1$ could be done with a similar calculation, but it requires bounding
$|\langle \Pi(f_1,f_2), f_3 w \rangle|$ for $f_3 \in L^{r_0'}(w)$ with
$$
 \sum_{R = I \times J} |a_R|\bla w \bra_R \Big| \Big \langle f_1, h_I \otimes \frac{1_J}{|J|}\Big \rangle \Big| \Big| \Big \langle f_2,  \frac{1_I}{|I|}\otimes h_J\Big\rangle \Big| \bla |f_3| \bra_{R}^w,
 $$
using the fuller strength of Corollary \ref{cor:main} and also Lemma \ref{lem:wm}.

\end{rem}

\section{Proof of Theorem \ref{thm:main3}}\label{sec:MulBloom}
We now move on to Theorem \ref{thm:main3} -- this directly forces us to deal with $m$-parameters, $m \ge 3$. Indeed, recall that 
\cite[Theorem 1.5]{EA} has no restrictions as long as the appearing singular integral operators are at most bi-parameter. 
Notice that it is obvious how $m$-parameter weights in the product space $\R^{d_1} \times \cdots \times \R^{d_m}$
are defined -- simply work with rectangles $R = I_1 \times \cdots \times I_m$. Notice also that if we set
$\calD = \prod_{i=1}^m \calD^i$, where $\calD^i$ is a dyadic grid on $\R^{d_i}$, and if we have a sequence
$A = (a_R)_{R \in \calD}$, the various product $\BMO$ norms of $A$ have a completely analogous definition in $m$-parameters.

We begin by noticing that a multi-parameter version of Theorem \ref{thm:main1} is true -- the proof is exactly the same.
\begin{thm}\label{thm:main1MUL}
For all $m$-parameter weights $w \in A_{\infty}$ on $\R^{d_1} \times \cdots \times \R^{d_m}$ we have
$$\|A\|_{\BMO_{\textup{prod}}} \sim \|A\|_{\BMO_{\textup{prod}, w}}.
$$
\end{thm}

To communicate the main idea regarding Theorem \ref{thm:main3},
including the definition of the space $\bmo^{\calI}(\nu)$ appearing in the theorem, we now restrict the discussion to the tri-parameter space
$\R^d = \R^{d_1} \times \R^{d_2} \times \R^{d_3}$. Let $\calD^i$ be a dyadic grid
in $\R^{d_i}$ and set $\calD = \calD^1 \times \calD^2 \times \calD^3$. We denote cubes in $\calD^i$ by $I_i, J_i, K_i$, etc.
It is obvious how to define, just like in the bi-parameter setting, the square functions $S_{\calD^1}^1$, $S_{\calD^1 \times \calD^2}^{1,2}$, $S_{\calD}$, and so on.
We now consider $b \colon \R^d \to \C$
and a tri-parameter Bloom weight $\nu = \mu^{1/p} \lambda^{-1/p}$, where $p \in (1, \infty)$ and $\mu, \lambda \in A_p(\R^d)$ are
tri-parameter $A_p$ weights.  In Section \ref{sec:BMO} we had a sequence $A = (a_R)$, but if we would have been talking about a function $b$ there, we would have simply considered the sequence $\big(\bla b, h_R \bra\big)$. For the commuting function $b$ we prefer to understand various $\BMO$ spaces
directly via the dualised forms \eqref{eq:standardH1BMO}. This is because it makes various relationships between the $\BMO$ spaces much more transparent.
 Here we follow \cite[Section 2]{EA}, where the facts
used in the following explanation are also proved.

Theorem \ref{thm:main3} is not dyadic -- thus, we need the following inequalities to hold uniformly for all dyadic grids $\calD^i$.
If for $i \in \{1,2,3\}$ we have
$$
|\langle b, f \rangle| \le C\|S_{\calD^i}^i f\|_{L^1(\nu)},
$$
we denote the optimal constant $C$ by $\|b\|_{\BMO^i(\nu)}$. Moreover, e.g. the norm $\|b\|_{\BMO^{1,2}_{\textup{prod}}(\nu)} = \|b\|_{\BMO^{1,2}(\nu)}$ has the obvious dual definition
as the best constant in the inequality
$$
|\langle b, f \rangle| \le \|b\|_{\BMO^{1,2}_{\textup{prod}}(\nu)} \|S_{\calD^1 \times \calD^2}^{1,2} f\|_{L^1(\nu)},
$$
or could be defined as in Section \ref{sec:BMO}. Similarly, it is obvious how to e.g. define $\|b\|_{\BMO^{1,3}_{\textup{prod}}(\nu)} = \|b\|_{\BMO^{1,3}(\nu)}$
and $\|b\|_{\BMO^{1,2, 3}_{\textup{prod}}(\nu)} = \|b\|_{\BMO^{1,2, 3}(\nu)}$. 

We can always estimate up by adding more parameters to square functions so that
e.g.
$$
\|S_{\calD^i}^i f\|_{L^1(\nu)} \lesssim \|S_{\calD^i \times \calD^j}^{i,j} f\|_{L^1(\nu)}.
$$
This explains the convenience of using the dual formulations here directly, as from such estimates
we can immediately see that e.g. $\BMO^1(\nu) \subset \BMO^{1,2}_{\textup{prod}}(\nu)$. Another thing is that, for instance, the $\BMO^1(\nu)$ condition implies that
\begin{equation}\label{eq:uniformBMO}
|\langle b(\cdot, x_2, x_3), g \rangle| \le \|b\|_{\BMO^1(\nu)} \|S_{\calD^1} g\|_{L^1(\nu(\cdot, x_2, x_3))}
\end{equation}
uniformly on (almost every) $x_2, x_3$, which is sometimes useful. 

Given a partition $\calI = \{\calI_1, \ldots, \calI_k\}$, $k \le 3$, of the parameter space $\{1, 2, 3\}$, we define
$$
\|b\|_{\bmo^{\calI}(\nu)} = \sup_{\bar u} \|b\|_{\BMO^{\bar u}(\nu)},
$$
where $\bar u = (u_i)_{i=1}^k$ is such that $u_i \in \calI_i$. We illustrate this with examples. If $k = 3$ we
look at $[T_1, [T_2, [T_3, b]]]$, where each $T_i$ is a one-parameter CZO on $\R^{d_i}$, and the related
condition is $\|b\|_{\bmo^{\{ \{1\}, \{2\},\{3\}\}}(\nu)} = \|b\|_{\BMO^{1,2, 3}(\nu)} = \|b\|_{\BMO^{1,2, 3}_{\textup{prod}}(\nu)}$.
If $k=2$ we may e.g have $\calI = \{ \{1,2\}, \{3\}\}$ and
look at $[T_1, [T_2, b]]$, where $T_1$ is a bi-parameter CZO on $\R^{d_1} \times \R^{d_2}$ and $T_2$ is a one-parameter CZO on
$\R^{d_3}$, and the related condition is $\|b\|_{\bmo^{\{\{1,2\},\{3\}\}}(\nu)} = \sup \{ \|b\|_{\BMO^{1, 3}(\nu)},  \|b\|_{\BMO^{2, 3}(\nu)}\}$.
Both the cases $k=3$ and $k=2$ are already covered by \cite[Theorem 1.5]{EA}, since all the CZOs are at most bi-parameter. The remaining case $k=1$
deals with the commutator $[b, T]$, where $T$ is a tri-parameter CZO on $\R^d$, and the related condition is  $\|b\|_{\bmo^{\{1,2,3\}}(\nu)} =
\sup_i \|b\|_{\BMO^i(\nu)}$.

The top-level strategy is always to first use the dyadic representation theorems \cite{Hy1, Hy2, Ma1, Ou} and reduce to commutators of dyadic model operators
in some fixed dyadic grids $\calD^i$. The difficulty levels of the different cases $k \in \{1,2,3\}$ are not strictly comparable. If $k=3$ we have the triple
commutator $[T_1, [T_2, [T_3, b]]]$, where the commutator structure -- and the related decomposition of the commutator -- is most complicated, but the operators themselves are one-parameter and the model operator decomposition \cite{Hy2} of each of them is the simplest. The remaining difficulty in \cite[Theorem 1.5]{EA} was \textbf{not} that
it was not known how to efficiently decompose arbitrary commutators $[T_1, [T_2, \ldots [b, T_k]]]$ or to use the complicated $\BMO$ spaces $\bmo^{\calI}(\nu)$.
The problem was that if one $T_i$ is of $k_i$-parameters, $k_i \ge 3$, then
the model operator decomposition \cite{Ou} produces a very complicated model operator --  a multi-parameter partial paraproduct with a $k$-parameter, $k \ge 2$, paraproduct component.

This is the problem we address in Theorem \ref{thm:main3}. In the proof we deal with the $k=1$ case in the tri-parameter situation. Moreover, we do not go through
all the cases  $[b,U]$, where $U$ is a tri-parameter model operator \cite{Ou}. Rather, we only focus on the new problem that $U = P$ for a tri-parameter partial paraproduct with a bi-parameter paraproduct component. It is then possible to take a completely general commutator $[T_1, [T_2, \ldots [b, T_k]]]$, follow the decomposition strategy of \cite{EA}
and handle the main new difficulty similarly as we do here. See also Section \ref{sec:indepth} to understand the role of Lemma \ref{lem:mainLem} and, thus, how everything depends
on Theorem \ref{thm:main1}. One more point: the main Lemma \ref{lem:mainLem} is applied to the \textbf{unweighted} product
$\BMO$ coefficients appearing in the dyadic model operators, and it is not related to the commuting function $b$. If we would e.g. deal with a $4$-parameter partial paraproduct with a tri-parameter paraproduct component, we would need
the full generality of Theorem \ref{thm:main1MUL} stated above. 

Due to the above discussion, we will now only
control $[b,P]$, where $P$ is a tri-parameter partial paraproduct with a bi-parameter paraproduct component.
We will now fix a 'little $\BMO$' function $b \in \bigcap_{i=1}^3 \BMO^i(\nu)$ with
the normalisation
$$
\sup_i \|b\|_{\BMO^i(\nu)} \le 1.
$$
Further, we will fix a tri-parameter partial paraproduct
$$
Pf = \sum_{ \substack{K_1 \\ I_1^{(i_1)} = J_1^{(j_1)} = K_1}} \sum_{K_2, K_3} a_{(K_j), I_1, J_1} \Big \langle
f, h_{I_1} \otimes h_{K_2} \otimes \frac{1_{K_3}}{|K_3|} \Big\rangle h_{J_1} \otimes \frac{1_{K_2}}{|K_2|} \otimes h_{K_3},
$$
where for all $K_1, I_1, J_1$ like above we have
$$
\big\|(a_{(K_j), I_1, J_1})_{K_2, K_3}\big\|_{\BMO_{\textup{prod}}^{2,3}} \le \frac{|I_1|^{1/2} |J_1|^{1/2}}{|K_1|}.
$$
We will show the tri-parameter Bloom estimate
\begin{equation}\label{eq:triBloom}
\| [b,P]f \|_{L^p(\lambda)} \lesssim \|f\|_{L^p(\mu)},
\end{equation}
where the implicit constant depends on the norms $[\mu]_{A_p}$ and $[\lambda]_{A_p}$, and we have polynomial dependency on the complexity, but this is not emphasised.

To show this, we will need some particular paraproducts. For $i \in \{1,2, 3\}$ we define
$$
A^i_1(b,f) = \sum_{I_i \in \calD^{i}} \Delta_{I_i}^i b \Delta_{I_i}^i f, \, \,
A^i_2(b,f) = \sum_{I_i \in \calD^{i}} \Delta_{I_i}^i b E_{I_i}^i f \,\, \textup{ and } \,\, A^i_3(b, f) = \sum_{I_i \in \calD^{i}} E_{I_i}^i b \Delta_{I_i}^i f.
$$
In one-parameter $\Delta_I g = \langle g, h_I \rangle h_I$ and $E_I g = 1_I \langle g \rangle_I$ are the usual martingale difference
and averaging operators, and then e.g. $\Delta_{I_1}^1 f(x) = (\Delta_{I_1} f(\cdot, x_2, x_3))(x_1)$.
For $i_1, i_2 \in \{1,2,3\}$ and $j_1, j_2 \in \{1, 2, 3\}$ define formally
$$
A_{j_1, j_2}^{i_1, i_2}(b,f) = A^{i_1}_{j_1}A^{i_2}_{j_2}(b,f)
$$ 
so that e.g.
$$
A_{1, 2}^{1,3}(b,f) = \sum_{I_3 \in \calD^{3}} A^1_1(\Delta_{I_3}^3 b, E_{I_3}^3 f) = \sum_{\substack{ I_1 \in \calD^1 \\ I_3 \in \calD^3}}\Delta_{I_1}^1 \Delta_{I_3}^3 b \Delta_{I_1}^1 E_{I_3}^3 f.
$$
What we need now is that according to \cite[Lemma 3.1]{EA} we have
$$
\|A_{j_1, j_2}^{i_1, i_2}(b,f)\|_{L^p(\lambda)} \lesssim \|f\|_{L^p(\mu)}
$$
as long as $(j_1, j_2) \ne (3,3)$. Indeed, for $j_1 \ne 3$ and $j_2 \ne 3$ we have
$$
\|A_{j_1, j_2}^{i_1, i_2}(b,f)\|_{L^p(\lambda)} \lesssim \|b\|_{\BMO^{i_1,i_2}_{\textup{prod}}(\nu)} \|f\|_{L^p(\mu)} \lesssim \|f\|_{L^p(\mu)},
$$
and if e.g. $j_1 = 3$ and $j_2 \ne 3$ we have
$$
\|A_{3, j_2}^{i_1, i_2}(b,f)\|_{L^p(\lambda)} \lesssim \|b\|_{\BMO^{i_2}(\nu)} \|f\|_{L^p(\mu)} \le \|f\|_{L^p(\mu)}.
$$

We now decompose the commutator $[b,P]$ using the paraproducts.
In $[b,P]f = bPf - P(bf)$ we decompose
$$
bPf = \sum_{j_1, j_2 = 1}^3 A_{j_1, j_2}^{1,3}(b, Pf) \qquad \textup{and} \qquad bf =  \sum_{j_1, j_2 = 1}^3 A_{j_1, j_2}^{1,2}(b, f).
$$
Because of the above paraproduct estimates and the weighted boundedness of $P$ (see e.g. \cite[Proposition 7.6]{HPW}),
to control $\| [b,P]f \|_{L^p(\lambda)}$ we only need to control
the $L^p(\lambda)$ norm of 
\begin{align*}
& A_{3,3}^{1,3}(b, Pf) - P(A_{3,3}^{1,2}(b,f)) \\
 &= \sum_{ \substack{K_1 \\ I_1^{(i_1)} = J_1^{(j_1)} = K_1}} \sum_{K_2, K_3} a_{(K_j), I_1, J_1} \Big \langle
f, h_{I_1} \otimes  h_{K_2} \otimes \frac{1_{K_3}}{|K_3|} \Big\rangle h_{J_1} \otimes \bla b \bra_{J_1 \times K_3}^{1,3}\frac{1_{K_2}}{|K_2|} \otimes h_{K_3} \\
&- \sum_{ \substack{K_1 \\ I_1^{(i_1)} = J_1^{(j_1)} = K_1}} \sum_{K_2, K_3} a_{(K_j), I_1, J_1} \Big \langle \bla b \bra_{I_1 \times K_2}^{1,2} \bla
f, h_{I_1} \otimes  h_{K_2} \bra_{1,2} \Big\rangle_{K_3}  h_{J_1} \otimes \frac{1_{K_2}}{|K_2|} \otimes h_{K_3}.
\end{align*}

In the first term we write $\bla b \bra_{J_1 \times K_3}^{1,3} = \big[\bla b \bra_{J_1 \times K_3}^{1,3} - \bla b \bra_{J_1 \times K_2 \times K_3}\big] + \bla b \bra_{J_1 \times K_2 \times K_3}$ and in the second term we write
$ \bla b \bra_{I_1 \times K_2}^{1,2} = \big[ \bla b \bra_{I_1 \times K_2}^{1,2} - \bla b \bra_{I_1 \times K_2 \times K_3}\big] + \bla b \bra_{I_1 \times K_2 \times K_3}$. We then combine the last two terms and further add and subtract $\bla b \bra_{K_1 \times K_2 \times K_3}$.
We begin by
dealing with one of the resulting terms
$$
E_1 := \sum_{ \substack{K_1 \\ I_1^{(i_1)} = J_1^{(j_1)} = K_1}} \sum_{K_2, K_3} \gamma_{(K_j), I_1, J_1}  \Big \langle
f, h_{I_1} \otimes  h_{K_2} \otimes \frac{1_{K_3}}{|K_3|} \Big\rangle h_{J_1} \otimes \frac{1_{K_2}}{|K_2|} \otimes h_{K_3},
$$
where $\gamma_{(K_j), I_1, J_1} =  a_{(K_j), I_1, J_1} \big[ \bla b \bra_{I_1 \times K_2 \times K_3} - \bla b \bra_{K_1 \times K_2 \times K_3}\big]$.
We write
$$
\bla b \bra_{I_1 \times K_2 \times K_3} - \bla b \bra_{K_1 \times K_2 \times K_3} = \sum_{l=1}^{i_1} \iint_{\R^{d_2} \times \R^{d_3}} \bla b, h_{I_1^{(l)}} \bra_1 \bla
h_{I_1^{(l)}} \bra_{I_1} \frac{1_{K_2}}{|K_2|}   \frac{1_{K_3}}{|K_3|}.
$$
Allowing a polynomial dependency on the complexity, we can fix $l$ and study the dualised form
\begin{align*}
\iint_{\R^{d_2} \times \R^{d_3}} \sum_{ \substack{K_1 \\ L_1^{(i_1-l)}  = K_1}} \sum_{ \substack{I_1^{(l)}  = L_1 \\ J_1^{(j_1)}  = K_1}} \sum_{K_2, K_3} 
|a_{(K_j), I_1, J_1}| |L_1|^{-1/2}& \big|\bla b, h_{L_1} \bra_1\big| 
\Big| \Big \langle f, h_{I_1} \otimes  h_{K_2} \otimes \frac{1_{K_3}}{|K_3|} \Big\rangle \Big| \\
& \Big| \Big \langle g, h_{J_1} \otimes  \frac{1_{K_2}}{|K_2|} \otimes h_{K_3} \Big\rangle\Big|
 \frac{1_{K_2}}{|K_2|}   \frac{1_{K_3}}{|K_3|}.
\end{align*}
Using \eqref{eq:uniformBMO} we reduce to
\begin{equation}\label{eq:eq1}
\begin{split}
\int_{\R^d} \Bigg(  \sum_{ \substack{K_1 \\ L_1^{(i_1-l)}  = K_1}} \frac{1_{L_1}}{|L_1|^2}
\otimes \Bigg[  \sum_{ \substack{I_1^{(l)}  = L_1 \\ J_1^{(j_1)}  = K_1}} \sum_{K_2, K_3} &
|a_{(K_j), I_1, J_1}| \Big| \Big \langle f, h_{I_1} \otimes  h_{K_2} \otimes \frac{1_{K_3}}{|K_3|} \Big\rangle \Big| \\
&\Big| \Big \langle g, h_{J_1} \otimes  \frac{1_{K_2}}{|K_2|} \otimes h_{K_3} \Big\rangle\Big|
 \frac{1_{K_2}}{|K_2|}   \frac{1_{K_3}}{|K_3|}
\Bigg]^2
\Bigg)^{1/2} \nu.
\end{split}
\end{equation}
A certain vector-valued inequality will now be derived using $A_{\infty}$ extrapolation, and the base case estimate for the extrapolation
will rely on Corollary \ref{cor:main}. 

\begin{lem}\label{lem:mainLem}
Suppose $0 < p, q < \infty$ and $w \in A_{\infty}(\R^{d_2} \times \R^{d_3})$ is a bi-parameter $A_{\infty}$ weight.
Suppose for all $k, j \in \N$ we are given a sequence $(a_{K_2, K_3}^{k,j})_{K_2, K_3}$ satisfying
$$
\big\|(a_{K_2, K_3}^{k,j})_{K_2, K_3}\big\|_{\BMO_{\textup{prod}}^{2,3}} \le 1.
$$
Then for all locally integrable $f_{j,k}, g_{j,k} \colon \R^{d_2} \times \R^{d_3} \to \C$ we have
\begin{align*}
& \iint \Bigg( \sum_j \Bigg[ \sum_k \sum_{K_2, K_3} |a_{K_2, K_3}^{k,j}| 
\Big| \Big\langle f_{j,k}, h_{K_2} \otimes \frac{1_{K_3}}{|K_3|} \Big\rangle \Big|
\Big| \Big\langle g_{j,k}, \frac{1_{K_2}}{|K_2|} \otimes h_{K_3} \Big\rangle \Big|
\frac{1_{K_2}}{|K_2|}   \frac{1_{K_3}}{|K_3|}
\Bigg]^p \Bigg)^{q/p} w \\
&\lesssim 
\iint \Bigg( \sum_j \Bigg[ \sum_k \Bigg( \sum_{K_2, K_3}
\bla \big|\bla f_{j,k}, h_{K_2}\bra_2 \big| \bra_{K_3}^2
\bla \big|\bla g_{j,k}, h_{K_3}\bra_3 \big| \bra_{K_2}^2
\frac{1_{K_2}}{|K_2|}   \frac{1_{K_3}}{|K_3|} \Bigg)^{1/2}
\Bigg]^p \Bigg)^{q/p} w.
\end{align*}
\end{lem}
\begin{proof}
By repeated use of $A_{\infty}$ extrapolation \cite[Theorem 2.1]{CUMP}, it is enough to prove that
\begin{align*}
& \iint \Bigg( \sum_{K_2, K_3} |a_{K_2, K_3}| 
\Big| \Big\langle f, h_{K_2} \otimes \frac{1_{K_3}}{|K_3|} \Big\rangle \Big|
\Big| \Big\langle g, \frac{1_{K_2}}{|K_2|} \otimes h_{K_3} \Big\rangle \Big|
\frac{1_{K_2}}{|K_2|}   \frac{1_{K_3}}{|K_3|} \Bigg) w \\
&\lesssim \iint \Bigg( \sum_{K_2, K_3}
\bla \big|\bla f, h_{K_2}\bra_2 \big| \bra_{K_3}^2
\bla \big|\bla g, h_{K_3}\bra_3 \big| \bra_{K_2}^2
\frac{1_{K_2}}{|K_2|}   \frac{1_{K_3}}{|K_3|}\Bigg)^{1/2} w.
\end{align*}
We note that the used extrapolation results are stated for the so-called Muckenhoupt bases and directly cover multi-parameter situations.
The remaining estimate now follows from Corollary \ref{cor:main} almost immediately.
\end{proof}
Applying the $p=2$ and $q=1$ case of the previous lemma with a fixed $x_1$ and $\nu(x_1, \cdot) \in A_{\infty}(\R^{d_2} \times \R^{d_3})$ to \eqref{eq:eq1} we reduce to
\begin{align*}
\int_{\R^d} \Bigg( &\sum_{ \substack{K_1 \\ L_1^{(i_1-l)}  = K_1}} \frac{1_{L_1}}{|L_1|^2 |K_1|^2} \otimes \Bigg[ \sum_{ \substack{I_1^{(l)}  = L_1 \\ J_1^{(j_1)}  = K_1}}
|I_1|^{1/2} |J_1|^{1/2}\\
&\Bigg( \sum_{K_2, K_3}
\bla \big|\bla f, h_{I_1} \otimes h_{K_2}\bra_{1,2} \big| \bra_{K_3}^2
\bla \big|\bla g, h_{J_1} \otimes h_{K_3}\bra_{1,3} \big| \bra_{K_2}^2
\frac{1_{K_2}}{|K_2|}   \frac{1_{K_3}}{|K_3|}\Bigg)^{1/2}
\Bigg]^2 \Bigg)^{1/2} \nu.
\end{align*}
We now estimate
\begin{align*}
&\Bigg( \sum_{K_2, K_3}
\bla \big|\bla f, h_{I_1} \otimes h_{K_2}\bra_{1,2} \big| \bra_{K_3}^2 
\bla \big|\bla g, h_{J_1} \otimes h_{K_3}\bra_{1,3} \big| \bra_{K_2}^2
\frac{1_{K_2}}{|K_2|}   \frac{1_{K_3}}{|K_3|}\Bigg)^{1/2} \\
&\le |I_1|^{1/2} |J_1|^{1/2}
\Bigg( \sum_{K_2, K_3}
\bla \big| \bla f, h_{K_2} \bra_2 \big| \bra_{I_1 \times K_3}^2  
\bla \big| \bla \Delta_{K_1,j_1}^1 g, h_{K_3} \bra_3 \big| \bra_{J_1 \times K_2}^2 \frac{1_{K_2}}{|K_2|}   \frac{1_{K_3}}{|K_3|}\Bigg)^{1/2}.
\end{align*}
Here $\Delta_{K_1,j_1}^1 g = \sum_{J_1^{(j_1)} = K_1} \Delta_{J_1}^1 g$ is a martingale block, and
we essentially threw away the cancellation in the $h_{I_1}$, which we can do as the nature of our argument produces one extra cancellation.
Using this we reduce to estimating
\begin{align*}
\int_{\R^d} M_{\calD^1}^1 \Bigg( \sum_{K_2} \frac{1_{K_2}}{|K_2|}& \otimes \big[M_{\calD^1 \times \calD^3}  \bla f, h_{K_2} \bra_2\big]^2\Bigg)^{1/2} \\
&\Bigg( \sum_{K_1} \Bigg[ M_{\calD^1}^1 \Bigg( \sum_{K_3} \big[M_{\calD^1 \times \calD^2} \bla \Delta_{K_1,j_1}^1 g, h_{K_3} \bra_3\big]^2 \otimes \frac{1_{K_3}}{|K_3|}  \Bigg)^{1/2}
\Bigg]^2 \Bigg)^{1/2} \nu.
\end{align*}
Using H\"older's inequality and the
standard estimates of Lemma \ref{lem:standardEst}, together with some vector-valued improvements of Lemma \ref{lem:standardEst}, which can be obtained by extrapolating the estimates of the same lemma, we derive the desired upper bound $\|f\|_{L^p(\mu)} \|g\|_{L^{p'}(\lambda^{1-p'})}$.
We have shown that
$$
|E_1| \lesssim \|f\|_{L^p(\mu)}.
$$
The error term with $\bla b \bra_{J_1 \times K_2 \times K_3} - \bla b \bra_{K_1 \times K_2 \times K_3}$ is handled similarly.

To complete the proof, we now deal with the error term
\begin{align*}
E_2 = \sum_{ \substack{K_1 \\ I_1^{(i_1)} = J_1^{(j_1)} = K_1}} \sum_{K_2, K_3} a_{(K_j), I_1, J_1} \Big \langle
f, h_{I_1}& \otimes  h_{K_2} \otimes \frac{1_{K_3}}{|K_3|} \Big\rangle \\
&  h_{J_1} \otimes \big[\bla b \bra_{J_1 \times K_3}^{1,3} - \bla b \bra_{J_1 \times K_2 \times K_3}\big]\frac{1_{K_2}}{|K_2|} \otimes h_{K_3},
\end{align*}
which is in some sense simpler than the error term $E_1$ and does not require the new techniques of this paper.
This is essentially because in the following expansion of the function $b$ we do not get a sum over the cubes in the shift parameter as above, but
rather over $J_2 \in \calD^2$. So we expand
$$
\big(\bla b \bra_{J_1 \times K_3}^{1,3} - \bla b \bra_{J_1 \times K_2 \times K_3}\big) 1_{K_2} = \sum_{J_2 \subset K_2} \Big \langle b, \frac{1_{J_1}}{|J_1|} \otimes
h_{J_2} \otimes \frac{1_{K_3}}{|K_3|} \Big\rangle h_{J_2}
$$
and dualise, and use this to reduce to
\begin{align*}
\iint_{\R^{d_1} \times \R^{d_3}} &\sum_{ \substack{K_1 \\ I_1^{(i_1)} = J_1^{(j_1)} = K_1}} \sum_{K_2, K_3} \sum_{J_2 \subset K_2}
|a_{(K_j), I_1, J_1}|  \frac{\big| \bla b, h_{J_2} \bra_2\big|}{|K_2|} \\
& \Big| \Big \langle f, h_{I_1} \otimes  h_{K_2} \otimes \frac{1_{K_3}}{|K_3|} \Big\rangle \Big|
 \big| \big \langle g, h_{J_1} \otimes  h_{J_2} \otimes h_{K_3} \big\rangle \big|
  \frac{1_{J_1}}{|J_1|} \frac{1_{K_3}}{|K_3|}.
\end{align*}
Using an estimate like \eqref{eq:uniformBMO} we get
\begin{equation}\label{eq:eq2}
\begin{split}
 \sum_{ \substack{K_1 \\ I_1^{(i_1)} = J_1^{(j_1)} = K_1}} \sum_{K_2, K_3} |a_{(K_j), I_1, J_1}|
 \Big| \Big \langle f, h_{I_1} \otimes & h_{K_2} \otimes \frac{1_{K_3}}{|K_3|} \Big\rangle \Big| \\
 &\bla \big( S_{\calD^2} \bla g, h_{J_1} \otimes h_{K_3} \bra_{1,3} \big) \bla \nu \bra_{J_1 \times K_3}^{1,3} \bra_{K_2}.
 \end{split}
\end{equation}
We now define the auxiliary operator
$$
Ug=\sum_{V_1, V_3} h_{V_1} 
\otimes 
\big(S_{\calD^2} \bla g, h_{V_1}\otimes h_{V_3}\bra_{1,3}\big)\bla \nu \bra_{V_1 \times V_3}^{1,3}  
\otimes
 h_{V_3},
$$
and notice that the term in \eqref{eq:eq2} equals
$$
\sum_{ \substack{K_1 \\ I_1^{(i_1)} = J_1^{(j_1)} = K_1}} \sum_{K_2, K_3} |a_{(K_j), I_1, J_1}|
\Big| \Big \langle f, h_{I_1} \otimes  h_{K_2} \otimes \frac{1_{K_3}}{|K_3|} \Big\rangle \Big|
\Big| \Big \langle Ug, h_{J_1} \otimes \frac{1_{K_2}}{|K_2|} \otimes h_{K_3} \Big\rangle \Big|.
$$
That is, using the original partial paraproduct $P$, we can see this as the pairing $\langle Pf, Ug \rangle$ with absolute values in (which is of no significance).
Thus, using the known weighted boundedness of $P$, we can simply dominate things by $\|f\|_{L^p(\mu)} \|Ug\|_{L^{p'}(\mu^{1-p'})}$.

It remains to show that $\|Ug\|_{L^{p'}(\mu^{1-p'})} \lesssim \|g\|_{L^{p'}(\lambda^{1-p'})}$. We will show this now.
Using again variants of the standard estimates of Lemma \ref{lem:standardEst} we get
\begin{align*}
\|Ug\|_{L^{p'}(\mu^{1-p'})} &\sim \Bigg\| \Bigg( \sum_{V_1, V_3} \frac{1_{V_1}}{|V_1|} \otimes \big[\big(S_{\calD^2} \bla g, h_{V_1}\otimes h_{V_3}\bra_{1,3}\big)\bla \nu \bra_{V_1 \times V_3}^{1,3}\big]^2 \otimes  \frac{1_{V_3}}{|V_3|} \Bigg)^{1/2} \Bigg\|_{L^{p'}(\mu^{1-p'})} \\
&\lesssim
 \Bigg\| \Bigg( \sum_{V_1, V_3} \frac{1_{V_1}}{|V_1|} \otimes \big(S_{\calD^2} \bla g, h_{V_1}\otimes h_{V_3}\bra_{1,3}\big)^2 \otimes  \frac{1_{V_3}}{|V_3|} \Bigg)^{1/2}\nu \Bigg\|_{L^{p'}(\mu^{1-p'})} \\
 &=
  \Bigg\| \Bigg( \sum_{V_1, V_3} \frac{1_{V_1}}{|V_1|} \otimes \big(S_{\calD^2} \bla g, h_{V_1}\otimes h_{V_3}\bra_{1,3}\big)^2 \otimes  \frac{1_{V_3}}{|V_3|} \Bigg)^{1/2} \Bigg\|_{L^{p'}(\lambda^{1-p'})} \\ &\lesssim \|g\|_{L^{p'}(\lambda^{1-p'})}.
\end{align*}
We have shown that
$
|E_2| \lesssim \|f\|_{L^p(\mu)}.
$
The error term with $\bla b \bra_{I_1 \times K_2}^{1,2} - \bla b \bra_{I_1 \times K_2 \times K_3}$ is handled similarly. This ends our proof of the tri-parameter Bloom estimate \eqref{eq:triBloom}.

 \end{document}